\documentclass[two side,reqno,11pt]{amsart}
\usepackage{amssymb,latexsym,xy,eucal,mathrsfs}
\usepackage{amssymb,amsmath,graphicx,colortbl,enumerate}
\usepackage{blindtext}
\usepackage{tikz}
\usetikzlibrary{decorations.markings}
\usepackage{pst-plot}
\usepackage{rotating}
\usepackage{pgfplots}
\usepackage{amsmath}
\usepackage{longtable,mathrsfs}
\usepackage[labelfont=bf,justification=raggedright,singlelinecheck=false]{caption}
\usepackage{caption}
\usepackage{subcaption}
\usepackage{framed}
\usepackage{float}
\numberwithin{equation}{section}
\newtheorem{thm}{Theorem}[section]

\newtheorem{theorem}[thm]{Theorem}
\newtheorem{proposition}[thm]{Proposition}
\newtheorem{example}[thm]{Example}
\newtheorem{remark}{Remark}[section]

\allowdisplaybreaks
\DeclareMathSizes{10}{9}{5}{5}   % For size 11 text
\usepackage{mathtools}

\begin{document}
\baselineskip 17truept
\title{Limit cycles of piecewise smooth differential systems with nilpotent center and linear saddle}
\date{}
\author{Nanasaheb Phatangare, Krishnat Masalkar and Subhash Kendre}
\address{Nanasaheb Phatangare, Department of Mathematics, Fergusson College,Pune,
	-411004, India.(M.S.)} \email{nmphatangare@gmail.com}
\address{Krishnat Masalkar, Department~of~Mathematics, Abasaheb ~Garware~ College, Pune-411004, India
	(M.S.)} \email{krishnatmasalkar@gmail.com}

\address{Subhash Kendre, Department~of~Mathematics, Savitribai~Phule~Pune~University, Pune-411007, India
	(M.S.)} \email{sdkendre@yahoo.com,drsubhash2010@gmail.com}

	\makeatletter
\@namedef{subjclassname@2020}{%
	\textup{2020} Mathematics Subject Classification}
\makeatother
	

\maketitle \noindent{ \small \textbf{Keywords}:  Piecewise linear differential system, Hamiltonian, limit cycle.
}

\begin{abstract}
	In this paper, we study the number of limit cycles of a piecewise smooth differential system separated by one or two parallel straight lines or rays formed by a nilpotent center or degenerate center and linear saddle. Piecewise linear differential systems separated by one or two parallel straight lines with one of the subsystems of type nilpotent center and other subsystems of type linear saddle can have at most two limit cycles and there are systems in these classes having one limit cycle. The limit cycle in particular consists of saddle separatrices of the subsystem.
\end{abstract}
\section { Introduction} \label{Ch5s1}
Limit cycles are isolated periodic orbits of the differential system. Study of limit cycles has long history \cite{poincare1891lintegration}. Many real-life phenomena are related to the existence of limit cycles \cite{van1920theory,van1926lxxxviii,zhabotinsky1964bperiodical,pechenkin2004understanding}. Piecewise linear differential systems separated by a straight line appear in mechanics, electrical circuits, economics, control theory, etc.\cite{teixeira2009perturbation,bernardo2008piecewise,simpson2010bifurcations} and study of a such system goes back to mid-twentieth century \cite{andronov1949theory}.
Continuous piecewise linear differential systems (PLDS) separated by one or two parallel straight lines appear in the control theory \cite{atherton200611,narendra2014frequency,henson1997nonlinear,llibre2014introduction}. Planar continuous piecewise linear vector fields with two zones are studied in \cite{freire1998bifurcation,nana2024limit}. A canonical form is obtained and different methods for the creation of periodic orbits are identified and their characteristics are noted. Global properties of continuous piecewise linear vector fields in $\mathbb{R}^2$ are studied in \cite{lum1991global,lum1992global}. Using the higher order averaging theory in \cite{llibre2017limit}, it is shown that the discontinuous quadratic and cubic polynomial perturbations of the linear center have more limit cycles than those of continuous and discontinuous linear perturbations.
Discontinuous PLDS formed by two linear differential systems separated by a straight line can have three limit cycles \cite{cardoso2020simultaneous,de2013limit,freire2014general,llibre2012three}.

In this paper, we discuss limit cycles of piecewise differential systems placed in two zones and systems in three zones. Limit cycles placed in two or three zones can be either sliding limit cycles or crossing limit cycles. The flow of the paper is as follows: In \textit{Section} \ref{Ch5s1}, normal forms of nilpotent center and linear saddle are presented. \textit{Section} \ref{Ch5s2} discusses limit cycles of a piecewise smooth system in two zones and three zones. In \textit{Section} \ref{Ch5s3}, limit cycles of piecewise smooth systems in two and three zones formed by the integrable degenerate center and Hamiltonian saddle are discussed. \textit{Section} \ref{Ch5s4} is dedicated to the piecewise systems separated by rays.

In \cite{llibre2021limit}, it is proved that the continuous piecewise differential system (PDS) separated by one straight line and composed of two linear saddles does not have limit cycles. Also, continuous PDS separated by two parallel straight lines and composed of three linear saddles does not have limit cycles. 
% In \cite{llibre2022limit} limit cycles of piecewise differential systems separated by one straight line and two straight lines and formed by linear Hamiltonian saddles and linear centers are studied. 
Here we state the results from \cite{llibre2022limit}.
\begin{proposition} Following statements hold for piecewise linear differential systems:
	\begin{enumerate}
		\item A continuous PLDS or discontinuous PLDS formed by one center and one linear Hamiltonian saddle and separated by one straight line has no limit cycles.
		\item A continuous PLDS formed by two centers and one linear Hamiltonian saddle and separated by two parallel straight lines has no limit cycles.
		\item A discontinuous PLDS formed by two centers and one linear Hamiltonian saddle and separated by two parallel straight lines can have at most one limit cycle. Moreover, there are systems in this class.
		\item A continuous PLDS formed by one center and two Hamiltonian saddles and separated by two parallel straight lines has no limit cycles.
		\item A discontinuous PLDS formed by one center and two Hamiltonian saddles and separated by two parallel straight lines can have at most one limit cycle. Moreover, there are systems in this class.
	\end{enumerate}
\end{proposition}
In \cite{phatangare2024limit}, limit cycles of PDS of the type nonlinear centre and saddle are discussed.
Here, we discuss the limit cycles of PDS placed in two zones and three zones and formed by the global nilpotent center and linear saddles.
%The homogeneous quadratic transformations occur in several practical problems including two-body interaction processes such as second-order chemical reactions, biological processes, asymptotic behavior of the errors of the methods in numerical analysis, etc.

The normal forms of nilpotent center at the origin are mentioned in the following result from \cite{dias2018polynomial}.
\begin{theorem}\label{Ch5th1} \cite{dias2018polynomial}
	Every degree three planar Hamiltonian polynomial vector field with a global nilpotent center at the origin, symmetric to the $x$-axis and with all infinite singular points being non-degenerated hyperbolic sectors, after a linear change of variables, can be written in one of the following forms:
	\begin{align}
		(\dot{x},\dot{y})=&( y, -x^3),\label{Ch5eq11}\\
		(\dot{x},\dot{y})=&(y+y^3,-x^3), \label{Ch5eq12}\\
		(\dot{x},\dot{y})=& (y+x^2y+ay^3,-x^3-xy^2) ~~~\text{with}~a\geq 0, \label{Ch5eq13}\\
		(\dot{x},\dot{y})=&(y-x^2y+ay^3, -x^3+xy^2)~ \text{with}~a\geq 1, \label{Ch5eq14}\\
		(\dot{x},\dot{y})=&(y+2xy+ax^2y+by^3,-x^3-y^2-axy^2)\nonumber\\
		&~\text{with}~ b>0~\text{and either}~a\geq 1,\nonumber\\
		&~\text{or}~ a<1~\text{with}~4(a-1)^2(a^3-a^2-ab-8a)-27b^2>0. \label{Ch5eq15} \end{align}
\end{theorem}

%\begin{framed}
	%\begin{figure}[H]
	%	\subfloat[Flow of system (\ref{Ch5eq11})]{\includegraphics[width=.25\linewidth]{Sys5.1.1-}}%\hspace{-1cm}
	%	\qquad
	%	\subfloat[Flow of system (\ref{Ch5eq12})]{\includegraphics[width=.25\linewidth]{Sys5.1.2-}} 		\qquad 		\subfloat[Flow of system (\ref{Ch5eq13})]{\includegraphics[width=.25\linewidth]{Sys5.1.3-}} 		\qquad 		\subfloat[Flow of system (\ref{Ch5eq14})]{\includegraphics[width=.25\linewidth]{Sys5.1.4-}}  		\qquad 		\subfloat[Flow of system (\ref{Ch5eq15})]{\includegraphics[width=.25\linewidth]{Sys5.1.5-}}	\\[-2ex] 			\vspace{1\baselineskip}   		\caption{\scriptsize{Normal forms of global nilpotent center at the origin.}} 	\end{figure}  \end{framed}

The normal form of the planar linear Hamiltonian saddle is proved in \cite{llibre2021limit}. Here, we state the result.
\begin{theorem}\label{Ch5th2}\cite{llibre2021limit}
	Let $\Delta=\alpha\delta-\beta^2<0, u=\beta\mu+\delta \gamma<0, v=\alpha\mu +\beta \gamma.$
	A Hamiltonian planar linear system having a saddle is topologically conjugate to
	\begin{align}\label{Ch5eq16}
		\dot{X}=\begin{cases}-\beta x-\delta y+\mu\\
			\alpha x+\beta y+\gamma \end{cases},
	\end{align}
	where $\alpha=0$ or $1$. Further, if $\alpha=0$ then $\gamma=0,\beta \neq 0$, and if $\alpha=1$ then $\delta <\beta^2$.\\ 
	Moreover, if the saddle point for the system (\ref{Ch5eq16}) is $(x_0,y_0)$, then points of intersection of its separatrices with the $y$-axis are $(0,A)$ and $(0,B)$, where
	\begin{align} \label{Ch5eqn1.7}
		x_0=-\frac{u}{\Delta},y_0=\frac{v}{\Delta},A=y_0+\dfrac{\beta +\sqrt{-\Delta}}{\delta}x_0, ~\text{and}~ B
		=  y_0+\dfrac{\beta -\sqrt{-\Delta}}{\delta}x_0.
	\end{align}
\end{theorem}
\begin{remark}
	Two separatrices of the system (\ref{Ch5eq16}) intersect the y-axis on the opposite side of the origin if and only if $A$ and $B$ have the opposite signs.
	Observe that, $A$ and $B$ have opposite signs if and only if $AB<0$, which amounts to say that $\dfrac{y_0}{x_0}<\beta^2-\alpha\delta+\dfrac{\beta}{\delta}.$
\end{remark} 
Hamiltonians for the systems (\ref{Ch5eq11})-(\ref{Ch5eq16}) are given by, respectively,
\begin{align}
	F_1(x,y)=&\dfrac{y^2}{2}+\dfrac{x^4}{4},\\
	F_2(x,y)=&\dfrac{y^2}{2}+\dfrac{y^4}{4}+\dfrac{x^4}{4},\\
	F_3(x,y)=&\dfrac{x^4}{4}+\dfrac{x^2y^2}{2}+\dfrac{y^2}{2}+\dfrac{ay^4}{4},\\
	F_4(x,y)=&\dfrac{x^4}{4}-\dfrac{x^2y^2}{2}+\dfrac{y^2}{2}+\dfrac{ay^4}{4},\\
	F_5(x,y)=&\dfrac{x^4}{4}+a\dfrac{x^2y^2}{2}+xy^2+\dfrac{y^2}{2}+\dfrac{by^4}{4},\\
	H(x,y)=&-\dfrac{1}{2}\alpha x^2-\dfrac{1}{2}\delta y^2-\beta xy -\gamma x+\mu y.
\end{align} 
Now consider a piecewise smooth differential system separated by a straight line;
\begin{align}
	\dot{X}=(\dot{x},\dot{y})=\begin{cases}
		({F}_{y}(x,y),-{F}_{x}(x,y))& \text{if}~x<0\\
		(H_y(x,y),-H_x(x,y))&\text{if}~x>0,
	\end{cases} \label{Ch5eq114}
\end{align}
where $F_{x}=\dfrac{\partial F}{\partial x}, F_{y}=\dfrac{\partial F}{\partial y}$ with $F=F_i$ for $i=1,2,3,4,5$.
\section{Piecewise smooth Hamiltonian systems formed by the nilpotent center and linear saddle.}\label{Ch5s2}
Limit cycles of a piecewise differential system placed in two zones and one in three zones formed by the linear center and Hamiltonian saddle are discussed in \cite{llibre2021limit,llibre2022limit}. In this section, we discuss the number of limit cycles of a piecewise smooth Hamiltonian differential system separated by one straight line formed by a nilpotent center and a Hamiltonian saddle and a system separated by two straight lines formed by a nilpotent center and two saddles. 
\begin{theorem}\label{Ch5th5.3}
	Let $k=a$ or $b$.
	\begin{enumerate}
		\item The system (\ref{Ch5eq114})  has a period annulus around the origin if and only if $\mu=0$ and $k\geq 0$.
		\item  The system (\ref{Ch5eq114}) has no limit cycle if and only if  $-\dfrac{\delta^2}{\mu^2}<k<0$. 
		\item The system (\ref{Ch5eq114}) has one limit cycle if and only if  $-\dfrac{\delta^2}{\mu^2}=k$.
		\item The system (\ref{Ch5eq114}) has two limit cycles if and only if  $k<-\dfrac{\delta^2}{\mu^2}$.
	\end{enumerate} 
\end{theorem}
\begin{proof}
	If the system (\ref{Ch5eq114}) has a periodic solution passing through $(0,y_1)$ and $(0,y_2)$ with $y_2<0<y_1$, then  $(0,y_1)$ and $(0,y_2)$ lies on the same level curve of the first integrals  $F=F(x,y)$ and $H=H(x,y)$. Hence, we have 
	$F(0,y_1)=F(0,y_2)  ~\text{and}~H(0,y_1)=H(0,y_2).$ 
	Therefore,
	\begin{align}\label{Ch5eq115}
		\left( y_{{1}}-y_{{2}} \right)  \left( y_{{1}}+y_{{2}} \right) 
		\left( k{y_{{1}}}^{2}+k{y_{{2}}}^{2}+2 \right)=0 
		~~\text{and}~\left( y_{{1}}-y_{{2}} \right)  \left( \delta\,y_{{1}}+\delta\,
		y_{{2}}-2\,\mu \right) 
		=0.
	\end{align}
	Since $y_1\neq y_2$, we have  
	\begin{align}\label{Ch5eq22}
		\left( y_{{1}}+y_{{2}} \right) 
		\left( k{y_{{1}}}^{2}+k{y_{{2}}}^{2}+2 \right)=0 
		~~\text{and}~  y_{{1}}+
		y_{{2}}=\frac{2\mu}{\delta}.
	\end{align}
	If $k\geq 0$,  then $y_1=-y_2$ and $\mu=0$. Hence, the system (\ref{Ch5eq114}) has a period annulus around the origin and is bounded by the separatrices of (\ref{Ch5eq16}).\\
	Now, assume that $k<0$ and that $\mu\neq 0$.
	Eliminating $y_2$ from the equations in (\ref{Ch5eq22}), we get
	\begin{align}\label{Ch5eq23}
		\frac{2\mu}{\delta}\left( y_1^2+\left( \frac{2\mu}{\delta}-y_1\right)^2+\frac{2}{k}\right)=0.
	\end{align}
	Solving (\ref{Ch5eq23}) for $y_1$, we get
	$$y_{{1}}={\frac {k\mu\pm\sqrt {-{k}^{2}{\mu}^{2}-k{\delta}^{2}}}{\delta\,
			k}}.
	$$
	Hence, if $-\dfrac{\delta^2}{\mu^2}<k<0$ then the system (\ref{Ch5eq114}) has no limit cycle, if $-\dfrac{\delta^2}{\mu^2}=k$ then system (\ref{Ch5eq114}) has one limit cycle, and if $k<-\dfrac{\delta^2}{\mu^2}$ then the system (\ref{Ch5eq114}) has two limit cycles.
\end{proof}
%\begin{framed} 	\begin{figure}[H]  		\subfloat[Flow of system (\ref{Ch5s5.18})]{\includegraphics[width=.29\linewidth]{Th521_01-}}%\hspace{-1cm}  		\qquad  		\subfloat[Flow of System (\ref{Ch5s5.19})]{\includegraphics[width=.29\linewidth]{Th521_02-}}  		\qquad   		\subfloat[Flow of system (\ref{Ch5s5.20})]{\includegraphics[width=.29\linewidth]{Th521_03-}}  		\qquad   		\subfloat[Flow of system (\ref{Ch5s5.21})]{\includegraphics[width=.29\linewidth]{Th521_04-}}   		\qquad   		\subfloat[Flow of system (\ref{Ch5s5.22})]{\includegraphics[width=.29\linewidth]{Th521_05-}}	\\[-2ex]     		\vspace{1\baselineskip}  		\caption{\scriptsize{Nilpotent center-linear saddle system (\ref{Ch5eq114}) when $\mu=0$.}} \label{Ch5Fig5.2}  	\end{figure}  \end{framed}

\begin{example}
	Consider the particular cases of system (\ref{Ch5eq114}) when $\mu=0$ and $k\geq 1.$
	\begin{enumerate}
		\item $F=F_1$:
		\begin{align}\label{Ch5s5.18}
			\dot{X}=\begin{cases}
				(y,-x^3),& \text{if}~x<0\\
				(x+y,-x-y-1),& \text{if}~x>0.
			\end{cases}
		\end{align}
		\item $F=F_2$:
		\begin{align}\label{Ch5s5.19}
			\dot{X}=\begin{cases}
				(y+y^3,-x^3),& \text{if}~x<0\\
				(x+y,-x-y-1),& \text{if}~x>0.
			\end{cases}
		\end{align}
		\item $F=F_3,a=1$:
		\begin{align}\label{Ch5s5.20}
			\dot{X}=\begin{cases}
				(y+x^2y+y^3,-x^3-xy^2),& \text{if}~x<0\\
				(x+y,-x-y-1),& \text{if}~x>0.
			\end{cases}
		\end{align}
		\item $F=F_4,a=1$:
		\begin{align}\label{Ch5s5.21}
			\dot{X}=\begin{cases}
				(y-x^2y+y^3,-x^3+xy^2),& \text{if}~x<0\\
				(x+y,-x-y-1),& \text{if}~x>0.
			\end{cases}
		\end{align}
		\item $F=F_5,a=b=1$:
		\begin{align}\label{Ch5s5.22}
			\dot{X}=\begin{cases}
				(y+2xy+x^2y+y^3,-x^3-y^2-xy^2),& \text{if}~x<0\\
				(x+y,-x-y-1),& \text{if}~x>0.
			\end{cases}
		\end{align}
	\end{enumerate}
\end{example}
From Part (1) of Theorem (\ref{Ch5th5.3}), each of the systems (\ref{Ch5s5.18})-(\ref{Ch5s5.22}) has a period annulus, see Figure (\ref{Ch5Fig5.2}).

\begin{framed}
	\begin{figure}[H]
		\subfloat[Flow of system (\ref{Ch5eq114}), $F=F_1,\mu\neq 0$]{\includegraphics[width=.29\linewidth]{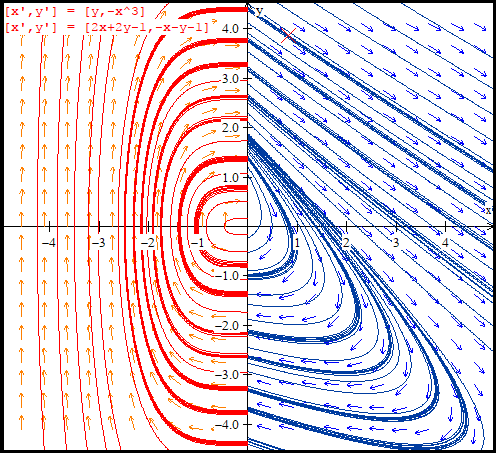}}%\hspace{-1cm}
		\qquad
		\subfloat[Flow of System (\ref{Ch5eq114}), $F=F_2,\mu\neq 0$]{\includegraphics[width=.29\linewidth]{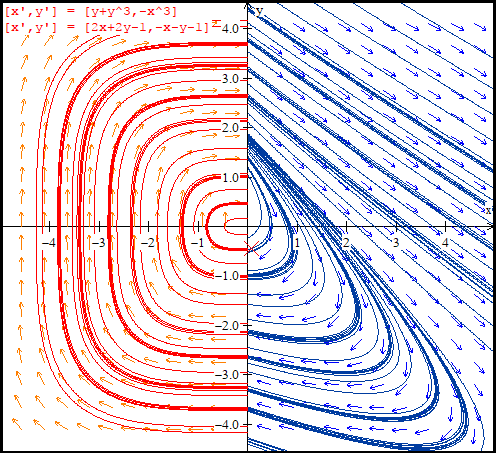}}
		\qquad
		\subfloat[Flow of system (\ref{Ch5s5.23})]{\includegraphics[width=.29\linewidth]{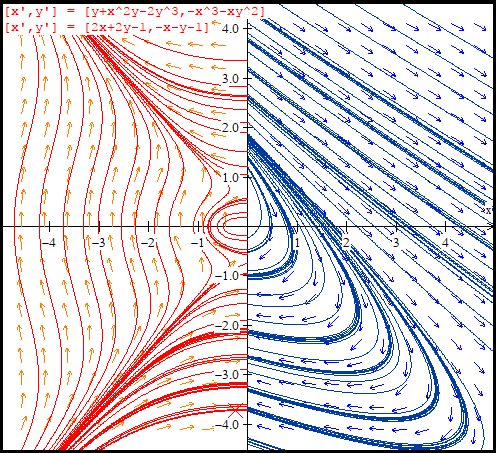}}
		\qquad
		\subfloat[Flow of system (\ref{Ch5s5.24})]{\includegraphics[width=.29\linewidth]{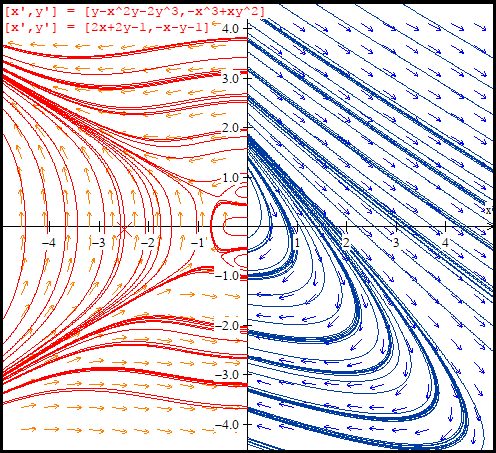}}
		\qquad
		\subfloat[Flow of system (\ref{Ch5s5.25})]{\includegraphics[width=.29\linewidth]{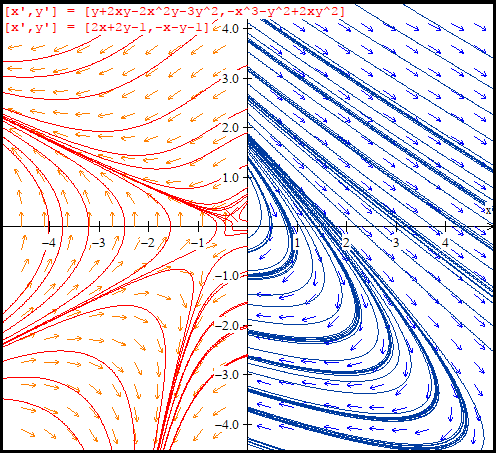}}	\\[-2ex]
		
		\vspace{1\baselineskip}
		\caption{\scriptsize{Nilpotent center-linear saddle system (\ref{Ch5eq114}) when $\mu\neq 0, -\frac{\delta^2}{\mu^2}<k<0$.}}\label{Ch5Fig5.3}
	\end{figure}
\end{framed}

\begin{example}
	Consider the system (\ref{Ch5eq114}) with $k=a$ or $b$ and $-\dfrac{\delta^2}{\mu^2}<k<0$. In this case, $\delta\neq 0$ and $\mu\neq 0$.
	\begin{enumerate}
		\item $F=F_3,a=-2$:
		\begin{align}\label{Ch5s5.23}
			\dot{X}=\begin{cases}
				(y+x^2y-2y^3,-x^3-xy^2),& \text{if}~x<0\\
				(2x+2y-1,-x-y-1),&\text{if}~x>0.
			\end{cases}
		\end{align}
		\item $F=F_4,a=-2$:
		\begin{align}\label{Ch5s5.24}
			\dot{X}=\begin{cases}
				(y-x^2y-2y^3,-x^3+xy^2),& \text{if}~x<0\\
				(2x+2y-1,-x-y-1),&\text{if}~x>0.
			\end{cases}
		\end{align}
		\item $F=F_5,a=-2, b=-3$:
		\begin{align}\label{Ch5s5.25}
			\dot{X}=\begin{cases}
				(y+2xy-2x^2y-3y^3,-x^3-y^2+2xy^2),& \text{if}~x<0\\
				(2x+2y-1,-x-y-1),&\text{if}~x>0.
			\end{cases}
		\end{align}
	\end{enumerate}
\end{example}
From Part (2) of Theorem (\ref{Ch5th1}), systems (\ref{Ch5s5.23})-(\ref{Ch5s5.25}) have no limit cycles.
\begin{example}
	Consider the system (\ref{Ch5eq114}) with $k=-\dfrac{\delta^2}{\mu^2}$.
	\begin{enumerate}
		\item $F=F_3,a=-4$:
		\begin{align}\label{Ch5s5.26}
			\dot{X}=\begin{cases}
				(y+x^2y-4y^3,-x^3-xy^2),&\text{if}~x<0\\
				(x+2y-1,x-y-3),&\text{if}~x>0.
			\end{cases}
		\end{align}
		\item $F=F_4,a=-4$:
		\begin{align}\label{Ch5s5.27}
			\dot{X}=\begin{cases}
				(y-x^2y-4y^3,-x^3+xy^2),&\text{if}~x<0\\
				(x+2y-1,x-y-3),&\text{if}~x>0.
			\end{cases}
		\end{align}
		\item $F=F_5,a=-4, b=-4$:
		\begin{align}\label{Ch5s5.28}
			\dot{X}=\begin{cases}
				(y+2xy-4x^2y-4y^3,-x^3-y^2+4xy^2),&\text{if}~x<0\\
				(x+2y-1,x-y-3),&\text{if}~x>0.
			\end{cases}
		\end{align}
	\end{enumerate}
\end{example}
From Part (3) of Theorem (\ref{Ch5th1}), systems (\ref{Ch5s5.26})-(\ref{Ch5s5.28}) can have at most one limit cycle.

%\begin{framed} 	\begin{figure}[H] 		\subfloat[Flow of system (\ref{Ch5s5.26})]{\includegraphics[width=.29\linewidth]{Th521_33-}}%\hspace{-1cm} 		\qquad  		\subfloat[Flow of system (\ref{Ch5s5.27})]{\includegraphics[width=.29\linewidth]{Th521_34-}} 		\qquad  		\subfloat[Flow of system (\ref{Ch5s5.28})]{\includegraphics[width=.29\linewidth]{Th521_35-}}		\\[-2ex]    		\vspace{1\baselineskip}  		\caption{\scriptsize{Nilpotent center-linear saddle system (\ref{Ch5eq114}) when $\mu\neq 0, -\frac{\delta^2}{\mu^2}=k<0$.}} \label{Ch5Fig5.4} 	\end{figure} \end{framed}

%\begin{framed} 	\begin{figure}[H]  		\subfloat[Flow of system (\ref{Ch5s5.29})]{\includegraphics[width=.29\linewidth]{Th521_43-}}%\hspace{-1cm}    		\qquad  		\subfloat[Flow of system (\ref{Ch5s5.30})]{\includegraphics[width=.29\linewidth]{Th521_44-}} 		\qquad  		\subfloat[Flow of system (\ref{Ch5s5.31})]{\includegraphics[width=.29\linewidth]{Th521_45-}}		\\[-2ex]
		%\label{Ch5Fig5.5}     		\vspace{1\baselineskip}    		\caption{\scriptsize{Nilpotent center-linear saddle system (\ref{Ch5eq114}) when $\mu\neq 0, k<-\dfrac{\delta^2}{\mu^2}$.}}  	\end{figure}  \end{framed}

\begin{example}
	Consider the system (\ref{Ch5eq114}) with $k<-\dfrac{\delta^2}{\mu^2}$.
	\begin{enumerate}
		\item $F=F_3,a=-5$:
		\begin{align}\label{Ch5s5.29}
			\dot(X)=\begin{cases}
				(y+x^2y-5y^3,-x^3-xy^2),&\text{if}~x<0\\
				(x+2y-1,x-y-3),&\text{if}~x>0.
			\end{cases}
		\end{align}
		\item $F=F_4,a=-5$:
		\begin{align}\label{Ch5s5.30}
			\dot(X)=\begin{cases}
				(y-x^2y-5y^3,-x^3+xy^2),&\text{if}~x<0\\
				(x+2y-1,x-y-3),&\text{if}~x>0.
			\end{cases}
		\end{align}
		\item $F=F_5,a=-5,b=-5$:
		\begin{align}\label{Ch5s5.31}
			\dot(X)=\begin{cases}
				(y+2xt-5x^2y-5y^3,-x^3-y^2+5xy^2),&\text{if}~x<0\\
				(x+2y-1,x-y-3),&\text{if}~x>0.
			\end{cases}
		\end{align}
	\end{enumerate}
\end{example}
From Part (4) of Theorem (\ref{Ch5eq114}), systems (\ref{Ch5s5.29})-(\ref{Ch5s5.31}) can have at most two limit cycles.

Now, consider the piecewise differential systems placed in three zones separated by two straight lines and formed by a nilpotent center and two Hamiltonian saddles,
\begin{align}\label{Ch5sys24}
	(\dot{x},\dot{y})=\begin{cases}
		\left({F}_{y}(x,y),-{F}_{x}(x,y)\right)& \text{if}~x<-1\\
		\left(H^{(1)}_y(x,y),-H^{(1)}_x(x,y)\right)&\text{if}~-1<x<1\\
		\left(H^{(2)}_y(x,y),-H^{(2)}_x(x,y)\right)&\text{if}~x>1,
	\end{cases} 
\end{align}
where $F_{x}=\dfrac{\partial F}{\partial x}, F_{y}=\dfrac{\partial F}{\partial y}$ with $F=F_i$ for $i=1,2,3,4,5$ and $H^{(j)}=-\frac{1}{2}\alpha_jx^2-\frac{1}{2}\delta_jy^2-\beta_jxy-\gamma_jx+\mu_jy$ for $j=1,2.$
\begin{theorem}
	Let $k=a$ or $b$.
	The system (\ref{Ch5sys24}) has at most two limit cycles. 
\end{theorem}
\begin{proof}
	Suppose that there is a periodic solution of the system (\ref{Ch5sys24}) that passes through the points $(-1, y_1),(-1,y_2), (1,y_3)$ and $(1,y_4)$ with $y_1<y_2$ and $y_4<y_3.$ Note that the solutions of Hamiltonian systems lie along the level curves of the Hamiltonian, we have
	\begin{eqnarray}
		F(-1, y_1)=&F(-1,y_2), \label{Ch5eq2.5}\\
		H^{(1)}(-1, y_2)=&H^{(1)}(1,y_3), \label{Ch5eq2.6}\\
		H^{(1)}(-1, y_1)=&H^{(1)}(1,y_4),\label{Ch5eq2.7}\\
		H^{(2)}(1, y_3)=&H^{(2)}(1,y_4).\label{Ch5eq2.8}
	\end{eqnarray}
	From the equation (\ref{Ch5eq2.5}) we get,
	\begin{equation} \label{Ch5eq2.9}
		\frac{1}{4}(y_1-y_2)(y_1+y_2)(2(a-1)-b(y_1^2-y_2^2))=0.
	\end{equation}
	Since $y_1\neq y_2$, equation (\ref{Ch5eq2.9}) gives,
	\begin{equation}\label{Ch5eq2.10}
		y_1=-y_2~~~~~~\text{or}~~~~~~b(y_1^2+y_2^2)=-2(a-1).
	\end{equation}
	But, since $b>0$ and $a\geq0$, $b(y_1^2+y_2^2)=-2(a-1)$ is not possible.\\
	From the equation (\ref{Ch5eq2.6}) and using $y_1=-y_2$ we get,
	\begin{equation}\label{Ch5eq2.11}
		\frac{\delta_1}{2} (y_1^2-y_3^2)+\beta_1(y_3-y_1)-\mu_1(y_3-y_1)+2\gamma_1=0.
	\end{equation}
	Similarly, the equation (\ref{Ch5eq2.7}) gives,
	\begin{align} \label{Ch5eq2.12}
		\frac{\delta_1}{2} (y_1^2-y_4^2)+\beta_1(y_1+y_4)+\mu_1(y_1-y_4)+2\gamma_1=0.
	\end{align}
	Further, from the equation (\ref{Ch5eq2.8}) and using $y_3\neq y_4$ we get,
	\begin{align}\label{Ch5eq2.14}
		\delta_2(y_3+y_4)=-2(\beta_2-\mu_2). 
	\end{align}
	Assume $\delta_1=0$. Then equations (\ref{Ch5eq2.11}) and (\ref{Ch5eq2.12}) become
	\begin{align}
		(\beta_1-\mu_1)(y_3-y_1)=&-2\gamma_1,~ \text{and}\label{Ch5eq214}\\ (\beta_1+\mu_1)y_1+(\beta_1-\mu_1)y_4=&-2\gamma_1,\label{Ch5eq215}
	\end{align} respectively.
	
	If $\beta_1-\mu_1\neq 0$, then from (\ref{Ch5eq214}) and (\ref{Ch5eq215}) we get, 
	$$y_3=-2\frac{\gamma_1}{\beta_1-\mu_1}+y_1,~\text{and}~~y_4=-2\frac{\gamma_1}{\beta_1-\mu_1}-\frac{\beta_1+\mu_1}{\beta_1-\mu_1}y_1,$$
	and $y_1=-y_2$ is a parameter.
	In this case, we have a period annulus around the origin of the system (\ref{Ch5sys24}).
	
	If $\beta_1=\mu_1\neq 0 $, then from (\ref{Ch5eq214}) we have $\gamma_1=0$ and $y_1=0,$  which is a contradiction, since $y_1\neq 0$. Hence, either $\beta_1\neq \mu_1$ or $\beta_1=\mu_1=0$ and we have a period annulus in this case.
	
	Next, assume that $\delta_1\neq 0$.
	From subtraction of equations (\ref{Ch5eq2.11}) and (\ref{Ch5eq2.12}) we get,
	\begin{equation}\label{Ch5eq2.16}
		\frac{1}{2}(y_3-y_4)(\delta_1(y_3+y_4)+2(\beta_1-\mu_1))=2(\beta_1+\mu_1)y_1.
	\end{equation}
	Multiplying equation (\ref{Ch5eq2.16}) by $\delta_2$ and using equation (\ref{Ch5eq2.14}), we get
	\begin{equation}\label{Ch5eq2.17}
		(y_3-y_4)(-\delta_1(\beta_2-\mu_2)+\delta_2(\beta_1-\mu_1))=2\delta_2(\beta_1+\mu_1)y_1.
	\end{equation}
	If $\delta_2=0$ then $\beta_2=\mu_2$. From the equation (\ref{Ch5eq2.11}) we can find $y_3$ in terms of $y_1$ and (\ref{Ch5eq2.12}) we can find $y_4$ in terms of $y_1$ and $y_1=-y_2$ is a parameter. Hence, in this case, we get that the system (\ref{Ch5sys24}) has a period annulus and at most one limit cycle formed by saddle separatrices.
	
	Now, assume that $\delta_1\delta_2\neq 0$.
	Let $\beta_i-\mu_i=l_i, i=1,2$, and $ \Delta=\delta_2l_1-\delta_1l_2$.
	
	Then $y_1, y_2,y_3$ and $y_4$ are related by the following equations,
	\begin{align}
		y_2=&-y_1,\\
		y_4=&-y_3-2\frac{l_2}{\delta_2},\\
		%&\Delta(y_3-y_4)-2l_1\delta_2y_1=0,~~ie.
		y_3=&\frac{\delta_2^2l_1y_1-\Delta l_2}{\delta_2 \Delta},\\
		\frac{\delta_1}{2}(y_1^2-y_3^2)+l_1(y_3-y_1)+2\gamma_1=&0.
	\end{align}
	Note that, $\Delta=0$ implies $l_1=0$. Hence, $l_2=0$ and form (\ref{Ch5eq2.14}), $\delta_2=0$, which is a contradicion.
	
	From the last two equations, we get a quadratic equation for $y_1$. It has at most two positive real roots and hence the system (\ref{Ch5sys24}) will have at most two limit cycles.
\end{proof}
\section{Limit cycles of piecewise smooth integrable systems} \label{Ch5s3}
In this section, we discuss the limit cycles of piecewise differential systems placed in two zones separated by one straight line and formed by an integrable degenerate center and Hamiltonian saddle and systems placed in three zones separated by two straight lines and formed by one integrable degenerate center and two Hamiltonian saddles.

%\begin{example}
The differential systems  
\begin{align}
	(\dot{x},\dot{y})&=
	\left( y(x^2-y^2)-2x^4y,x(x^2+y^2)-2x^3y^2\right), \label{Ch5eq3.1}\\
	(\dot{x},\dot{y})&=
	\left( -y(3x^2+y^2),x(x^2-y^2)\right) \label{Ch5eq3.2}
\end{align}
are integrable systems with degenerate centers at $(0,0)$. Integrals of the systems (\ref{Ch5eq3.1}) and (\ref{Ch5eq3.2}) are given by
\begin{align}
	I_1(x,y)=&\ln(x^2+y^2-1)-\frac{1}{2}\ln(x^4+y^4)-\tan^{-1}\left(\frac{x^2}{y^2} \right),\label{Ch5eq3.3}\\
	I_2(x,y)=&\frac{1}{2}\ln(x^2+y^2)-\frac{x^2}{x^2+y^2},\label{Ch5eq3.4}
\end{align}
respectively.

%\end{example}
Consider the system
\begin{align}\label{Ch5eq3.5}
	(\dot{x},\dot{y})=\begin{cases}
		({I_1}_{y}(x,y),-{I_1}_{x}(x,y))& \text{if}~x<0\\
		(H_y(x,y),-H_x(x,y))&\text{if}~x>0.
	\end{cases} 
\end{align}
\begin{theorem} \label{Ch5th5.5}
	Let $x_0=-\dfrac{\beta \mu+\delta \gamma}{\alpha \delta-\beta^2}>0$.
	The system (\ref{Ch5eq3.5}) has at most one limit cycle which consists of saddle separatrices of the right subsystem if and only if $\mu=0.$ 
\end{theorem}
\begin{proof}
	Suppose that there is a periodic solution of the system (\ref{Ch5eq3.5})  passing through $(0,y_1)$ and $(0,y_2)$ with $y_1<y_2$. Then we have 
	$H(0,y_1)=H(0,y_2)$ and $I_1(0,y_1)=I_1(0,y_2).$
	This implies, 
	\begin{align}\label{Ch5eq3.6}
		\ln(y_1^2-1)-\ln(y_1^2)&=\ln(y_2^2-1)-\ln(y_2^2),~\text{and}\\
		\delta(y_1+y_2)&=2\mu.
	\end{align}
	Assume that $\delta\neq 0$ and $k=\dfrac{2\mu}{\delta}$. Then, 
	\begin{align}\label{Ch5eq3.8}
		(y_1^2-1)(y_2^2)&=(y_2^2-1)(y_1^2),~ \text{and}~~
		y_1=k-y_2.
	\end{align}
	Therefore, we get $y_1=- y_2=k-y_2$. This gives $k=\mu=0$.
	Hence, the system (\ref{Ch5eq3.5}) has a period annulus consisting of periodic orbits passing through $(0,y)$ for all $y>1$ if and only if $\mu=0$. 
\end{proof}
\begin{example}
	Consider the system (\ref{Ch5eq3.5}) with $\beta=1,\delta=2,\mu=0,\alpha=-1$ and $\gamma=3$;
	\begin{align}\label{Ch5s5.58}
		\dot{X}=\begin{cases}
			(y(x^2-y^2)-2x^4y,x(x^2+y^2)-2x^3y^2),&\text{if}~x<0\\
			(-x-2y,-x+y+3),&\text{if}~x>0.
		\end{cases}
	\end{align}
	From Theorem (\ref{Ch5th5.5}), system (\ref{Ch5s5.58}) has at most one limit cycle consisting of saddle separatrices.
	
\end{example}
\begin{framed}
\begin{figure}[H]
	\subfloat[Flow of system (\ref{Ch5eq3.1})]{\includegraphics[width=.29\linewidth]{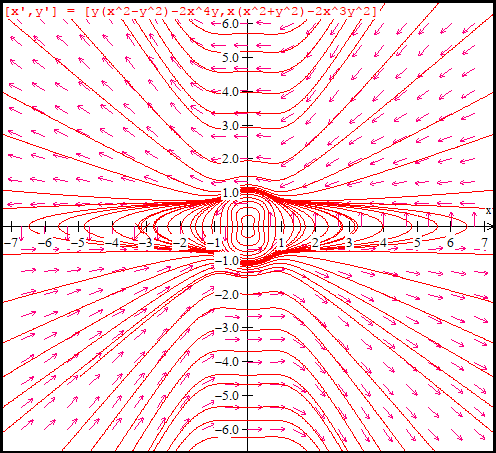}}%\hspace{-1cm}
	\qquad
	\subfloat[Linear saddle]{\includegraphics[width=.29\linewidth]{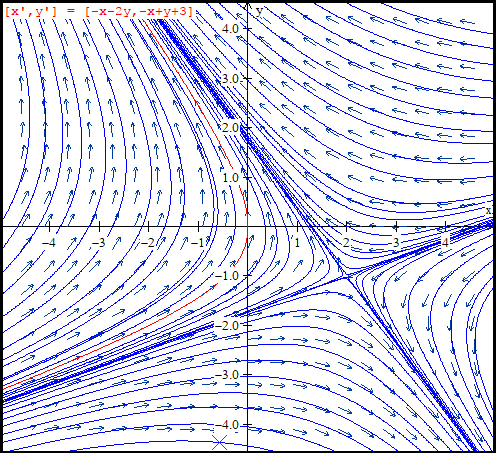}}
	\qquad
	\subfloat[Flow of system (\ref{Ch5s5.58})]{\includegraphics[width=.29\linewidth]{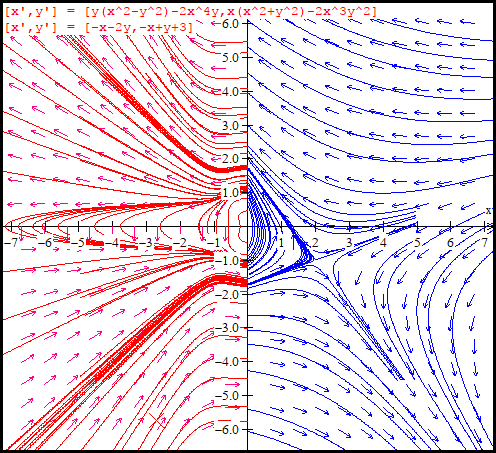}}		\\[-2ex]
	
	\vspace{1\baselineskip}
	\caption{Integrable center-linear saddle system (\ref{Ch5s5.58}).} \label{Ch5Fig5.6}
\end{figure}
\end{framed}
Now consider the piecewise differential systems formed by a degenerate center and Hamiltonian saddle
\begin{align}\label{Ch5eq3.12}
	(\dot{x},\dot{y})=\begin{cases}
		({I_2}_{y}(x,y),-{I_2}_{x}(x,y))& \text{if}~x<0\\
		(H_y(x,y),-H_x(x,y))&\text{if}~x>0.
	\end{cases} 
\end{align}
\begin{theorem} \label{Ch5th5.6}
	Let $x_0=-\dfrac{\beta \mu+\delta \gamma}{\alpha \delta-\beta^2}>0$.
	The system (\ref{Ch5eq3.12}) has at most one limit cycle which consists of saddle separatrices of the right subsystem if and only if $\mu=0.$ 
\end{theorem}
\begin{proof}
	Let $A$ and $B$ be the constants given by the equation (\ref{Ch5eqn1.7}). Assume that there is a periodic solution of the system (\ref{Ch5eq3.12})  passing through the points $(0,y_1)$ and $(0,y_2)$ with $y_1<y_2$. Then,  
	$H(0,y_1)=H(0,y_2)$ and $I_2(0,y_1)=I_2(0,y_2).$
	
	Hence, 
	\begin{align*}
		\ln(y_1^2)&=\ln(y_2^2)\Rightarrow y_1=-y_2,~\text{and}\\
		\delta(y_1+y_2)&=2\mu \Rightarrow \mu=0.
	\end{align*}
	Therefore, for each $|y|<\min\{A, B\}$, there is a periodic orbit passing through the points $(0,\pm y)$. Thus, there is a period annulus inside a limit cycle formed by saddle separatrices of the system (\ref{Ch5eq3.12})  if and only if $\mu=0$.
\end{proof} 
\begin{example}
	Consider a center-saddle system (\ref{Ch5eq3.12}) with $\beta=1,\delta=2,\mu=0,\alpha=-1$ and $\gamma=3$;
	\begin{align}\label{Ch5s5.60}
		\dot{X}=\begin{cases}
			(-y(3x^2+y^2),x(x^2-y^2)),&\text{if}~x<0\\
			(-x-2y,-x+y+3),&\text{if}~x>0.
		\end{cases}
	\end{align} 
	In this case, the center is integrable and degenerate and the saddle on the right side is linear Hamiltonian with saddle point $(x_0,y_0)$ such that $x_0>0$.
	
	From Theorem (\ref{Ch5th5.6}), system (\ref{Ch5s5.60}) can have at most one limit cycle consisting of separatrices of the saddle.

\begin{framed}
	\begin{figure}[H]
		\subfloat[Flow of system (\ref{Ch5eq3.2})]{\includegraphics[width=.29\linewidth]{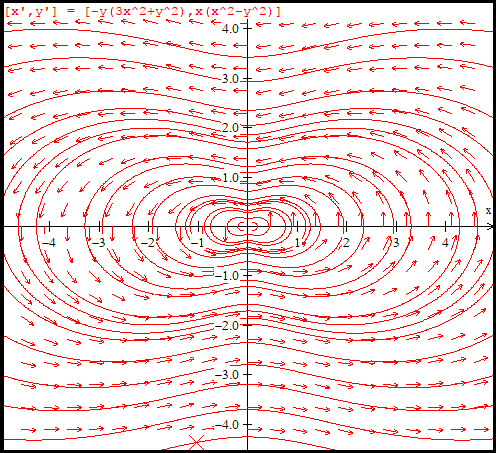}}%\hspace{-1cm}
		\qquad \qquad
		\subfloat[Center-saddle (\ref{Ch5s5.60})]{\includegraphics[width=.29\linewidth]{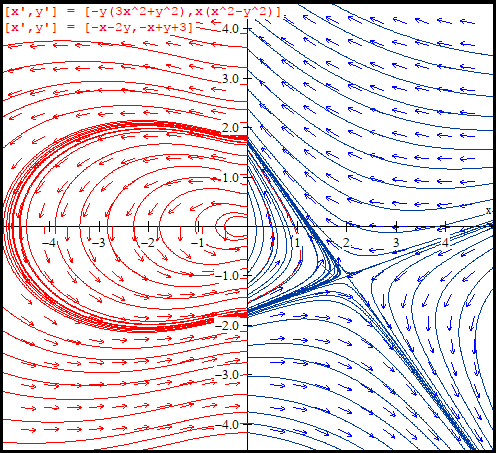}}
		\\[-2ex]
		
		\vspace{1\baselineskip}
		\caption{\scriptsize{Integrable degenerate center-linear saddle}} \label{Ch5Fig5.7}
	\end{figure}
\end{framed}

\end{example}
Next, consider the piecewise differential system placed in three zones formed by one degenerate center (\ref{Ch5eq3.1}) and two Hamiltonian saddles,
\begin{align}\label{Ch5sys3.13}
	(\dot{x},\dot{y})=\begin{cases}
		\left({I_1}_{y}(x,y),-{I_1}_{x}(x,y)\right)& \text{if}~x<-1\\
		\left(H^{(1)}_y(x,y),-H^{(1)}_x(x,y)\right)&\text{if}~-1<x<1,\\
		\left(H^{(2)}_y(x,y),-H^{(2)}_x(x,y)\right)&\text{if}~x>1
	\end{cases} 
\end{align} 
\begin{theorem}\label{Ch5th3.3} Consider the systems (\ref{Ch5sys3.13}).
	\begin{enumerate}
		\item If $\delta_1 \delta_2\neq 0$, then
		the system (\ref{Ch5sys3.13}) has at most one periodic solution. 
		\item If $\delta_1=0,~ \delta_2(\mu_1-\beta_1)\neq 0$, then the system (\ref{Ch5sys3.13}) has at most one limit cycle.
		\item If $\delta_1\neq 0$ and $\delta_2=0$, then the system (\ref{Ch5sys3.13}) has at most one limit cycle.
		\item If $\delta_2=\delta_1=\mu_2-\beta_2=0$, then the system (\ref{Ch5sys3.13}) has a period annulus.
	\end{enumerate}
\end{theorem}
\begin{proof}
	Suppose that there is a periodic solution of the system (\ref{Ch5sys3.13}) which passes through the points $(-1, y_1),(-1,y_2), (1,y_3)$ and $ (1,y_4)$ with $y_1<y_2$ and $y_4<y_3.$ Since solutions of the integrable systems lie along level curves of the first integrals, we have
	\begin{eqnarray}
		I_1(-1, y_1)=I_1(-1,y_2), \label{Ch5eq3.14}\\
		H^{(1)}(-1, y_2)=H^{(1)}(1,y_3), \label{Ch5eq3.15}\\
		H^{(1)}(-1, y_1)=H^{(1)}(1,y_4),\label{Ch5eq3.16}\\
		H^{(2)}(1, y_3)=H^{(2)}(1,y_4).\label{Ch5eq3.17}
	\end{eqnarray}
	From equation (\ref{Ch5eq3.14}), we get
	\begin{equation} \label{Ch5eq3.18}
		2\ln  \left( {\frac {y_{{1}}}{y_{{2}}}} \right) -\frac{1}{2}\ln  \left( {
			\frac {{y_{{1}}}^{4}+1}{{y_{{2}}}^{4}+1}} \right) -\arctan \left( {y_{
				{1}}}^{-2} \right) +\arctan \left( {y_{{2}}}^{-2} \right) =0.
	\end{equation}
	Real roots of the equation (\ref{Ch5eq3.18}) satisfy
	$y_1^2=y_2^2.$ But $y_1\neq y_2$. Hence, $y_1=-y_2.$
	
	From the equation (\ref{Ch5eq3.15}) and (\ref{Ch5eq3.16}) we have
	\begin{equation}\label{Ch5eq3.19}
		\frac{\delta_1}{2} (y_2^2-y_3^2)+\beta_1(y_2+y_3)+\mu_1(y_2-y_3)+2\gamma_1=0.
	\end{equation}
	Similarly, from the equation (\ref{Ch5eq3.16}) we have 
	\begin{align}\label{Ch5eq3.20}
		\frac{\delta_1}{2} (y_1^2-y_4^2)+\beta_1(y_1+y_4)+\mu_1(y_1-y_4)+2\gamma_1=0.
	\end{align}
	Also, from the equation (\ref{Ch5eq3.17}) we have
	\begin{equation}\label{Ch5eq3.21}
		y_3=y_4-2\nu_2.
	\end{equation}
	Subtracting equation (\ref{Ch5eq3.20}) from the equation (\ref{Ch5eq3.19}) and in the view of (\ref{Ch5eq3.21}) we get,
	\begin{align}\label{Ch5eq3.22}
		y_1=-y_2=-\frac{(y_4-\nu_2)(\nu_1-\nu_2)}{l_1},
	\end{align}
	where $l_i=\dfrac{\mu_i+\beta_i}{\delta_i}, \nu_i=\dfrac{\mu_i-\beta_i}{\delta_i},k_i=\dfrac{\gamma_i}{\delta_i} ~~\text{for}~~i=1,2.$\\
	Now, eliminating $y_1$ from equations (\ref{Ch5eq3.19}), (\ref{Ch5eq3.20}) and (\ref{Ch5eq3.21}) we get,
	\begin{align}\label{Ch5eq3.23}
		Py_4^2+Qy_4+R=0,
	\end{align}
	where 
	$P=\nu_2^2-l_1^2, Q=-2(\nu_1+\nu_2)P$ and $R=\nu_2^4+\nu_1^2\nu_2^2+P\nu_1\nu_2+(4k_1-2\nu_2^2)l_1^2.$
	
	The equation (\ref{Ch5eq3.23}) has at most one positive or negative root which gives the value of $y_4$ and the values of $y_1,y_2, y_3$, satisfying equations (\ref{Ch5eq3.14}) to (\ref{Ch5eq3.17}), can be determined from the equations (\ref{Ch5eq3.21}) and (\ref{Ch5eq3.22}). Hence, if $\delta_1\delta_2(\mu_1+\beta_1)\neq 0$ and if $\delta_1\neq 0$, then the system (\ref{Ch5sys3.13}) has atmost one periodic solution. \\
	Now consider the case that $\delta_1=0$. From the equations (\ref{Ch5eq3.19}) and (\ref{Ch5eq3.20}) we get,
	\begin{align}\label{Ch5eq3.24}
		-&(\beta_1+\mu_1)y_1+(\beta_1-\mu_1)y_3+2\gamma_1=0,~\text{and}\\ \label{Ch5eq3.25}
		&(\beta_1+\mu_1)y_1+(\beta_1-\mu_1)y_4+2\gamma_1=0.
	\end{align}
	Further, if $\beta_1-\mu_1\neq 0$  then from equations (\ref{Ch5eq3.24}) and (\ref{Ch5eq3.25}) we get,
	\begin{align*}
		y_3=&\frac{-2\gamma_1}{\beta_1-\mu_1}+\dfrac{\beta_1+\mu_1}{\beta_1-\mu_1}y_1,~\text{and}\\
		y_4=&\frac{-2\gamma_1}{\beta_1-\mu_1}-\dfrac{\beta_1+\mu_1}{\beta_1-\mu_1}y_1.
	\end{align*}
	Substituting $y_3$ and $y_4$ in (\ref{Ch5eq3.18}), we get only one value for $y_1$ when $\delta_2\neq 0$ and hence, only one periodic solution.
	
	If $\delta_1=0,~\delta_2=0$ and $\mu_2-\beta_2=0$, then $y_1$ and $y_2$ determined in terms of parameters $y_3$ and $y_4$. Hence, in this case, we get a period annulus. Similarly, if $\delta_2=0$ and $\delta_1\neq 0$ then the system has at most one periodic solution. 
\end{proof}
Now, consider the piecewise differential system in three zones formed by a degenerate center (\ref{Ch5eq3.2}) and two Hamiltonian saddles,
\begin{align}\label{Ch5sys3.26}
	(\dot{x},\dot{y})=\begin{cases}
		\left({I_2}_{y}(x,y),-{I_2}_{x}(x,y)\right)& \text{if}~x<-1,\\
		\left(H^{(1)}_y(x,y),-H^{(1)}_x(x,y)\right)&\text{if}~-1<x<1,\\
		\left(H^{(2)}_y(x,y),-H^{(2)}_x(x,y)\right)&\text{if}~x>1.
	\end{cases} 
\end{align} 
\begin{theorem} Consdier the system (\ref{Ch5sys3.26}).
	\begin{enumerate}
		\item If $\delta_1 \delta_2\neq 0$ then
		the system (\ref{Ch5sys3.26}) has at most one periodic solution. 
		\item If $\delta_1=0,~ \delta_2(\mu_1-\beta_1)\neq 0$ then the system (\ref{Ch5sys3.26}) has one limit cycle.
		\item If $\delta_1\neq 0$ and $\delta_2=0$ then the system (\ref{Ch5sys3.26}) has at most one limit cycle.
		\item If $\delta_2=\delta_1=\mu_2-\beta_2=0$ then the system (\ref{Ch5sys3.26}) has a period annulus.
	\end{enumerate}
\end{theorem}
\begin{proof}
	Suppose that there is a periodic solution of the system (\ref{Ch5sys3.26}) which passes through the points $(-1, y_1),(-1,y_2), (1,y_3)$ and $(1,y_4)$ with $y_1<y_2$ and $y_4<y_3.$ Since, solutions of the integrable systems lie along level curves of the first integrals, we have
	\begin{eqnarray}
		I_2(-1, y_1)=I_2(-1,y_2), \label{Ch5eq1.27}\\
		H^{(1)}(-1, y_2)=H^{(1)}(1,y_3), \label{Ch5eq1.28}\\
		H^{(1)}(-1, y_1)=H^{(1)}(1,y_4),\label{Ch5eq1.29}\\
		H^{(2)}(1, y_3)=H^{(2)}(1,y_4),\label{Ch5eq1.30}
	\end{eqnarray}
	From the equation (\ref{Ch5eq1.27}) we get,
	\begin{equation} \label{Ch5eq3.31}
		\ln\left(\frac{y_1^2+1}{y_2^2+1}\right)-\frac{2}{y_1^2+1}+\frac{2}{y_2^2+1}=0.
	\end{equation}
	The real roots of the equation satisfy
	$y_1^2=y_2^2.$ But $y_1\neq y_2$ hence we get $y_1=-y_2.$
	%Thus from the equations (\ref{Ch5eq1.28}), (\ref{Ch5eq1.29}) and (\ref{Ch5eq1.30}), we get
	%\begin{align}
	%\delta_1(y_3^2-y_1^2)-2(\beta_1-\mu_1)y_3-2(\beta_1+\mu_1)y_1+4\gamma_1=0\\
	%\delta_1(y_4^2-y_1^2)-2(\beta_1-\mu_1)y_4+2(\beta_1+\mu_1)y_1+4\gamma_1=0\\
	%\delta_2(y_3+y_4)+(\beta_2-\mu_2)=0.
	%\end{align}
	The rest of the proof is similar to the proof of Theorem \ref{Ch5th3.3}.
\end{proof}
\section{Limit cycles of piecewise system separated by rays}\label{Ch5s4}
In this section, we discuss the piecewise differential system separated by rays and formed by the linear integrable system with saddle and center. We convert systems into polar form and study the number of limit cycles for a piecewise smooth differential system of type center-saddle separated by rays $\theta=k$ and center-saddle-saddle type separated by two rays $\theta_i=k_i~~\text{for}~~i=1,2$.

A linear integrable  system with saddle  point at $(\alpha, \beta), \alpha>0$ and oriented in anticlockwise direction is given by
\begin{align}\label{Ch5sys4.1}
	(\dot{x},\dot{y})=(a (x-\alpha)-(y-\beta), -(x-\alpha)+a (y-\beta),-1<a<0.
\end{align} 
Its integral is given by,
\begin{align}
	a\ln\left(\frac{y-x-\beta+\alpha}{y+x-\alpha-\beta}\right)+\ln\left((y-\beta)^2-(x-\alpha)^2\right)=\ln c^2.
\end{align}
Therefore, the solution passing through $(\rho, 0)$ is 
\begin{align}
	a\ln\left(\frac{y-x+\alpha-\beta}{y+x-\alpha-\beta}\right)+&\ln\left((y-\beta)^2-(x-\alpha)^2\right)\nonumber\\
	&=a\ln\left(\frac{-\rho+\alpha-\beta}{\rho-\alpha-\beta}\right)+\ln\left(\beta^2-(\rho-\alpha)^2\right).
\end{align} 
This solution intersects $y$-axis (i.e., $\theta=\frac{\pm \pi}{2}$) at  two points $(0,y_1)$ and $(0,y_2)$ which will satisfy 
\begin{align}\label{Ch5eq4.4}
	a\ln\left(\frac{y_i+\alpha-\beta}{y_i-\alpha-\beta}\right)+\ln\left((y_i-\beta)^2-\alpha^2\right)=a\ln\left(\frac{-\rho+\alpha-\beta}{\rho-\alpha-\beta}\right)+\ln\left(\beta^2-(\rho-\alpha)^2\right)
\end{align}
for $i=1,2$. Here, note that $y_1\neq -y_2$. 

Simlarly, for any fix $0<\theta=\phi<\frac{\pi}{2}$, the solution given by (\ref{Ch5eq4.4}), intersects two rays $\theta=\phi$ and $\theta=-\phi$ which are not symmetric about the $x$-axis.

The saddle separatrices, invariant eigenspaces of the system (\ref{Ch5sys4.1}), are 
\begin{align}
	\frac{y-\beta}{x-\alpha}=\pm 1.
\end{align}
The saddle separatrices intersect rays $\theta=\pm \phi$ at points that are symmetric to the $x$-axis.

Polar form of the general planar  system $\dot{x}=F(x,y),~~ \dot{y}=G(x,y)$ is given by
$$\frac{dr}{d \theta}=r\frac{F(r\cos \theta, r\sin \theta) \cos \theta +G(r\cos \theta, r\sin \theta)\sin \theta}{-F(r\cos \theta, r\sin \theta)\sin \theta+G(r\cos \theta, r\sin \theta)\cos \theta}.$$
Hence, polar form of the integrable systems (\ref{Ch5eq3.1})  and (\ref{Ch5eq3.2}) are given by, respectively,
\begin{align}\label{Ch5eq4.6}
	\frac{dr}{d\theta}&=-2r(r^2-1)\frac{\cos^3\theta \sin\theta}{\cos^4\theta+\sin^4\theta},~\text{and}\\
	\label{Ch5eq4.7}
	\frac{dr}{d\theta}&=4r \frac{\cos^3\theta \sin\theta}{\cos^4\theta-4\cos^2\theta \sin^2\theta-\sin^4\theta}.
\end{align}
Integrals of the equations (\ref{Ch5eq4.6}) and  (\ref{Ch5eq4.7})  are
\begin{align}
	\frac{r^2-1}{r^2}&=\frac{1}{2\cos^4(\theta)-2\cos^2(\theta)+1}e^{-2\tan^{-1}(2\cos^2(\theta)-1)},~\text{and}\\
	r^4&=\frac{1}{4\cos^4(\theta)-2\cos^2(\theta)-1}e^{\frac{\sqrt{5}}{10}\tanh^{-1}\left(\frac{\sqrt{5}}{10}(8\cos^2(\theta)-2)\right)},
\end{align}
respectively.

Let $n\in \mathbb{N}$.
The generalized trigonometric functions $x(\theta)=Cs(\theta), y(\theta)=Sn(\theta)$ satisfy the properties mentioned in  Proposition (\ref{Ch5pr4.1}) \cite{alvarez2006generating}. See \cite{gasull1998center,liapunov2016stability} for proof.
\begin{proposition}\label{Ch5pr4.1}
	\begin{enumerate}
		\item $Cs^{2n}(\theta)+nSn^2(\theta)=1$.
		\item $Cs(\theta)$ and $Sn(\theta)$ are $T$-period functions with $\displaystyle T=2\sqrt{\frac{\pi}{n}} \frac{\Gamma(\frac{1}{2n})}{\Gamma(\frac{n+1}{2n})}$, where $\Gamma$ denotes the Gamma function.
		\item $\displaystyle \int_{0}^{T}Sn^p(\theta) Cs^q(\theta) d\theta=0$ if $p$ or $q$ is odd.
		\item $\displaystyle \int_{0}^{T}Sn^p(\theta) Cs^q(\theta) d\theta=\frac{2}{\sqrt{n^{p+1}}}\frac{\Gamma(\frac{p+1}{2})\Gamma(\frac{q+1}{2n})}{\Gamma(\frac{p+1}{2}+\frac{q+1}{2n})}$ if both $p$ and $q$ are even.
	\end{enumerate}
\end{proposition}
Consider the system $(\dot{x},\dot{y})=(-y, x^{2n-1})$ with nilpotent center at the origin. If we substitute $x=RCs(\theta), y=RSn(\theta)$, then in these coordinates the system becomes
\begin{align*}
	\dot{R}&=\frac{x^{2n-1}(-y)+y(x^{2n-1})}{R^{2n-1}},~~\dot{\theta}=\frac{x(x^{2n-1})-ny(-y)}{R^{n+1}}
\end{align*}
which is equivalent to $ (\dot{R},\dot{\theta})=(0,1).$

Let $-1<a<0$ and let $0<\phi\leq \frac{\pi}{2}$.
Consider the piecewise smooth planar differential system separated by rays $\theta=\phi$ and $\theta=-\phi$ passing through the origin,
\begin{align}\label{Ch5sys4.10}
	(\dot{x},\dot{y})&=(a (x-\alpha)-(y-\beta),-(x-\alpha)+a (y-\beta)), &\text{if}~-\phi<\theta<\phi\\
	\label{Ch5sys4.11}
	(\dot{R},\dot{\theta})&=(0,0), ~&\text{if}~\phi<\theta<2\pi-\phi,
\end{align}
where $(r,\theta)$ are polar coordinates and $(R,\theta)$ are the generalized polar coordinates.
\begin{theorem}\label{Ch5th5.9}
	The piecewise smooth differential system formed by (\ref{Ch5sys4.10}) and (\ref{Ch5sys4.11}) has exactly one limit cycle which consists of saddle separatrices of (\ref{Ch5sys4.10}) and solution of (\ref{Ch5sys4.11}) passing through the point of intersection of saddle separatrix and ray $\theta=\phi$.
\end{theorem}
\begin{proof}
	The only saddle separatrix of the system (\ref{Ch5sys4.10}) is symmetric about the $x$-axis. Also, all solutions of the system (\ref{Ch5sys4.11}) are symmetric and the system (\ref{Ch5sys4.11}) has a global center. Hence, the proof.
\end{proof}
Let $0<\phi\leq \dfrac{\pi}{2}$ and $-1<a<0$.
Consider a piecewise smooth planar differential system separated by rays $\theta=-\phi$ and $\theta=\phi$ passing through the origin,
\begin{align}\label{Ch5sys4.14}
	(\dot{x},\dot{y})&=
	\begin{cases}
		(a (x-\alpha)-(y-\beta), ~~-(x-\alpha)+a (y-\beta)), ~~&\text{if}~~-\phi<\theta<\phi\\
		(y(x^2-y^2)-2x^4y, x(x^2+y^2)-2x^3y^2), ~~~&\text{if}~~\phi<\theta<2\pi-\phi.
	\end{cases}
\end{align}
\begin{theorem}
	The piecewise-smooth system (\ref{Ch5sys4.14}) has one limit cycle if $|\alpha\pm\beta|\leq 1$, otherwise no limit cycle.
\end{theorem}
\begin{proof}
	The proof is similar to the proof of Theorem \ref{Ch5th5.9}.
\end{proof}

\begin{framed}
	\begin{figure}[H]
		\subfloat[System \ref{Ch5sys4.14} when $\phi=\frac{\pi}{2}$]{\includegraphics[width=.31\linewidth]{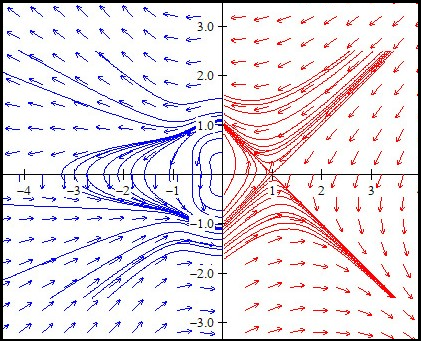}}%\hspace{-1cm}
		\qquad \qquad
		\subfloat[System \ref{Ch5sys4.15} when $\phi=\frac{\pi}{2}$]{\includegraphics[width=.3\linewidth]{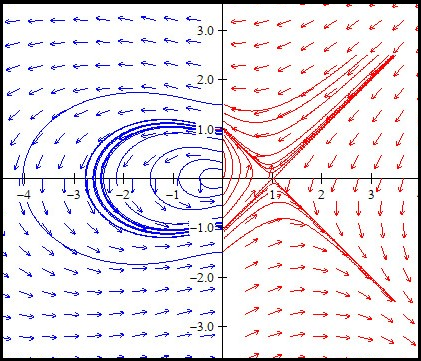}}
		\\[-2ex]
		\vspace{1\baselineskip}
		\caption{\scriptsize{Integrable piecewise system separated by rays}}
	\end{figure}
\end{framed}

Let $0<\phi\leq \dfrac{\pi}{2}$ and $-1<a<0$.
Consider the piecewise smooth planar differential system separated by rays $\theta=-\phi$ and $\theta=\phi$ passing through the origin,
\begin{align}\label{Ch5sys4.15}
	(\dot{x},\dot{y})&=
	\begin{cases}
		(a (x-\alpha)-(y-\beta), ~~-(x-\alpha)+a (y-\beta)), ~~&\text{if}~~-\phi<\theta<\phi\\
		(-y(3x^2+y^2), x(x^2-y^2)), ~~&\text{if}~~\phi<\theta<2\pi-\phi.
	\end{cases}
\end{align}
\begin{theorem}
	The piecewise-smooth system (\ref{Ch5sys4.15}) has exactly one limit cycle.
\end{theorem}
\begin{proof}
	The proof is similar to the proof of Theorem \ref{Ch5th5.9}.
\end{proof}
\section{{Acknowledgement}}
The authors would like to express sincere gratitude to the reviewers for their valuable suggestions and comments.
\section{Conflict of interest statement}
No funding was received for conducting this study. The authors have no financial or proprietary interests in any material discussed in this article.

%\bibliographystyle{plain}
%\bibliography{ref3.bib}

\end{document}